\documentclass[12pt]{amsart}

\usepackage{amssymb}
\usepackage{graphicx}
\usepackage{enumerate}
\usepackage{multirow}
\usepackage{amsmath,color}
\usepackage{hyperref}
\usepackage{url}
\usepackage[section]{placeins}

\newtheorem{thm}{Theorem}[section]
\newtheorem{cor}[thm]{Corollary}
\newtheorem{lem}[thm]{Lemma}
\newtheorem{prop}[thm]{Proposition}

\newtheorem{rem}[thm]{Remark}

\newtheorem{Observation}[thm]{Observation}
\numberwithin{equation}{section}

\makeatletter
\@namedef{subjclassname@2020}{%
  \textup{2020} Mathematics Subject Classification}
\makeatother

\title{Distance-regular graphs admitting a perfect $1$-code}

\author{Mojtaba Jazaeri}
\address{Department of Mathematics, Shahid Chamran University of Ahvaz, Ahvaz, Iran}
\email{M.Jazaeri@scu.ac.ir, M.Jazaeri@ipm.ir}
\begin{document}



\begin{abstract}
In this paper, we study the problem that which of distance-regular graphs admit a perfect $1$-code. Among other results, we characterize distance-regular line graphs which admit a perfect $1$-code. Moreover, we characterize all known distance-regular graphs with small valency at most $4$, the distance-regular graphs with known putative intersection arrays for valency $5$, and all distance-regular graphs with girth $3$ and valency $6$ or $7$ which admit a perfect $1$-code.
\end{abstract}

\maketitle


\section{Introduction}
It is well known that the classical coding theory studies perfect codes in Hamming graphs and these graphs are distance-transitive. In $1973$, Biggs \cite{Biggs} initiated an investigation of perfect codes in distance-transitive graphs. Since distance-transitive graphs are a family of distance-regular graphs, it is reasonable to study perfect codes in distance-regular graphs. Neumaier \cite{N} introduced the notion of a completely regular code and proved that a perfect code in a distance-regular graph is indeed a completely regular code. We refer to the monograph \cite[Chap.~11]{BCN} for more background on perfect codes in distance-regular graphs. In this paper, we study the problem that which of distance-regular graphs admit a perfect $1$-code. In some literature, an efficient domination set is used instead of a perfect $1$-code (see for example \cite{CMO} and \cite{CMOS}). For an overview on recent progress of this topic, we refer to \cite{CMOS}. We first state some observations and equations on perfect $1$-codes and then we characterize distance-regular line graphs which admit a perfect $1$-code. Furthermore, we state some facts about perfect codes in antipodal distance-regular graphs and give an overview on perfect $1$-codes in distance-regular graphs with small diameter at most $4$. Moreover, we characterize all known distance-regular graphs with small valency at most $4$, the distance-regular graphs with known putative intersection arrays for valency $5$, and all distance-regular graphs with girth $3$ and valency $6$ or $7$ which admit a perfect $1$-code.
\section{Preliminaries}
In this paper, all graphs are undirected and simple, i.e., there are no loops or multiple edges. Moreover, we consider the eigenvalues of the adjacency matrix of a graph. A connected graph $\Gamma$ is called distance-regular with diameter $d$ and intersection array
\begin{equation*}
\{b_{0},b_{1},\ldots,b_{d-1};c_{1},c_{2},\ldots,c_{d}\}
\end{equation*}
whenever for each pair of vertices $x$ and $y$ at distance $i$, where $0 \leq i \leq d$, the number of neighbours of $x$ at distance $i+1$ and $i-1$ from $y$ are constant numbers $b_{i}$ and $c_{i}$, respectively. This implies that a distance-regular graph is regular with valency $b_{0}=k$ and the number of neighbours of $x$ at distance $i$ from $y$ is a constant number $k-b_{i}-c_{i}$ which is denoted by $a_{i}$. A $k$-regular graph with $n$ vertices is called strongly regular with parameters $(n,k,\lambda,\mu)$ whenever the number of common neighbours of two adjacent vertices is $\lambda$ and the number of common neighbours of two non-adjacent vertices is $\mu$. Note that for a strongly regular graph to be of diameter $2$ and thus a distance-regular graph, it needs to be connected and non-complete. Moreover, for a distance-regular graph with diameter $d$, the number of vertices at distance $i$ from an arbitrary given vertex is constant and denoted by $K_{i}$. Furthermore, $$K_{i+1}=\frac{K_{i}b_{i}}{c_{i+1}},$$ where $i=0,1,\ldots,d-1$ and $K_{0}=1$.

Recall that a projective plane of order $q$ is a point-line incidence structure such that each line has $q+1$ points, each point is on $q+1$ lines, and every pair of points is on a unique line. Furthermore, the incidence graph of a projective plane is a bipartite distance-regular graph with diameter three and intersection array $\{q+1,q,q;1,1,q+1\}$. Moreover, the distinct eigenvalues (of the adjacency matrix) of this graph are $\{\pm(q+1),\pm \sqrt{q}\}$.

Let $\Gamma$ be a graph with vertex set $V$. Then any subset $C$ of $V$ is called a code in $\Gamma$. Let $\overline{\Gamma_{t}(c)}$ denote the set of vertices at distance at most $t$ from $c$, where $c \in C$. Then the code $C$ in $\Gamma$ is called perfect $t$-code whenever $\{\overline{\Gamma_{t}(c)} \mid c \in C\}$ is a partition of the vertex set $V$ (cf. \cite{K}). This implies that a code $C$ is perfect $1$-code whenever $C$ is an independent set and every vertex outside $C$ has a unique neighbour in $C$. Furthermore, it is trivial to see that if $C_{1}$ and $C_{2}$ are two perfect $1$-codes in a graph, then $|C_{1}|=|C_{2}|$ since there exists a bijection between $C_{1}$ and $C_{2}$ by the definition of a perfect $1$-code. Moreover, if $C$ is a perfect $1$-code in a $k$-regular graph $\Gamma$, then
\begin{equation} \label{equation}
|C|=\frac{|V|}{k+1},
\end{equation}
because $\{\Gamma_{1}(c)\mid c \in C\}$ is a partition for the vertex set $V$ and each part has size $k+1$.
\begin{rem} \label{completegraph}
Let $\Gamma$ be a regular graph with vertex set $V$. Then $\Gamma$ admits an one-element subset of $V(\Gamma)$ as a perfect $1$-code if and only if $\Gamma$ is a complete graph.
\end{rem}
The following two observations are trivial by the definition of a perfect $1$-code but are useful.
\begin{Observation} \label{Observation1}
Let $C$ be a perfect $1$-code with at least two elements in a connected graph and $x,y \in C$. Then $d(x,y) \geq 3$. Moreover, there exist at least two elements at distance $3$ in $C$. To see this let $x \in C$ and $y$ be a vertex at distance $2$ from $x$ which is indeed outside $C$. Then there exists a unique element $z \in C$ which is adjacent to $y$ and therefore the distance between $x$ and $z$ is $3$.
\end{Observation}
\begin{Observation} \label{Observation2}
Let $C$ be a perfect $1$-code in a connected regular graph with vertex set $V$ and valency $k$. Then $\{C,V \backslash C\}$ is an equitable partition with the quotient matrix $$\begin{bmatrix}0&k\\1&k-1\end{bmatrix}.$$ Therefore $-1$ must be an eigenvalue of (the adjacency matrix of) this graph.
\end{Observation}
\subsection{Completely regular codes}
Let $\Gamma$ be a connected regular graph with vertex set $V$ and a code $C$, where $|C|\geq 2$. Then the number
\begin{equation*}
d(C):=\min\{d(x,y)\mid x,y \in C,x\neq y\}
\end{equation*}
is called the minimum distance of $C$. The distance $v \in V$ from $C$ is defined by
\begin{equation*}
d(v,C):=\min\{d(v,w)\mid w \in C\}
\end{equation*}
and the number
\begin{equation*}
t(C):=\max\{d(v,C)\mid v \in V\}
\end{equation*}
is called the covering radius of $C$. Let
\begin{equation*}
C_{\ell}:=\{v \in V\mid d(v,C)=\ell\},
\end{equation*}
where $\ell=0,1,\ldots,t(C)$. Then the code $C$ is called completely regular whenever for all $\ell$, every vertex in $C_{\ell}$ has the same number $c_{\ell}$ of neighbours in $C_{\ell-1}$, the same number $b_{\ell}$ of neighbours in $C_{\ell+1}$ and the same number $a_{\ell}$ of neighbours in $C_{\ell}$.  It is trivial to see that every one-element code is completely regular in a distance-regular graph. This definition was fist introduced by Neumaier \cite{N}. He proved that a code $C$ in a distance-regular graph is a perfect code if and only if it is a completely regular code with $d(C)=2t(C)+1$ (cf. \cite[Thm.~4.3]{N}). It follows that if $C$ is a perfect $1$-code in a distance-regular graph, then the code $C$ is a completely regular code with $d(C)=3$ since $t(C)=1$.
\section{Antipodal distance-regular graphs} \label{secantipodal}
Let $\Gamma$ be an antipodal distance-regular graph with diameter $d\geq 3$. Then the folded graph of $\Gamma$ which is denoted by $\overline{\Gamma}$ is a graph whose vertex set is the fibers and two fibers are adjacent whenever there exists an edge between them in the graph $\Gamma$. Recall that two fibers have the same size and if two are adjacent in $\overline{\Gamma}$, then there exists a perfect matching between them in the graph $\Gamma$. Moreover, each vertex in one fiber is adjacent to at most one vertex in another fiber since each pair of vertices in a fiber is at distance $d\geq 3$. Let $C$ be a perfect code in $\Gamma$. Then $C$ is a disjoint union of some fibers (cf. \cite[last Remark on p.~349]{BCN}). It follows that if $\overline{C}$ is the corresponding code in the folded graph $\overline{\Gamma}$, then $\overline{C}$ is a perfect code in $\overline{\Gamma}$. Therefore we have the following proposition.
\begin{prop} \label{antipodal}
Let $\Gamma$ be an antipodal distance-regular graph with diameter $d\geq 3$. Then the graph $\Gamma$ admits a perfect code $C$ if and only if $C$ is a disjoint union of some fibers such that $\overline{C}$ is a perfect code in the folded graph $\overline{\Gamma}$.
\end{prop}
Recall that the folded graph of an antipodal distance-regular graph with diameter $3$ is a complete graph. Therefore we can conclude the following corollary about a perfect $1$-code in such a graph.
\begin{cor} \label{antipodaldrg3}
Let $\Gamma$ be an antipodal distance-regular graph with diameter $3$. Then a code $C$ is $1$-perfect if and only if $C$ is a fiber.
\end{cor}
Furthermore, the folded graph of an antipodal distance-regular graph with diameter $d=4$ or $5$ is a strongly regular graph and therefore we can conclude the following corollary about a perfect $1$-code in these graphs since there is no perfect $1$-code in a strongly regular graph.
\begin{cor} \label{antipodaldrg45}
There is no perfect $1$-code in an antipodal distance-regular graph with diameter $d=4$ or $5$.
\end{cor}
We note that the folded graph of the Doubled odd graph $\operatorname{DO_{n}}$ with diameter $2n-1$ is the Odd graph $\operatorname{O_{n}}$ with diameter $n-1$ and therefore we can conclude the following corollary.
\begin{cor} \label{DoubleOdd}
The Doubled odd graph $\operatorname{DO_{n}}$ admits a perfect $1$-code if and only if the Odd graph $\operatorname{O_{n}}$ admits a perfect $1$-code.
\end{cor}
\section{Distance-regular line graphs}
In this section, we characterize distance-regular line graphs which admit a perfect $1$-code. We denote the line graph of a graph $\Gamma$ by $\operatorname{L(\Gamma)}$. The main theorem of this section is as follows.
\begin{thm}
Let $\Gamma$ be a distance-regular graph with least eigenvalue $-2$. Then $\Gamma$ admits a perfect $1$-code if and only if $\Gamma$ is isomorphic to one of the following graphs.
\begin{itemize}
  \item The cycle graph $C_{6n}$,
  \item the line graph of the Petersen graph,
  \item the line graph of the Tutte-Coxeter graph.
\end{itemize}
\end{thm}
We first observe that if $\Gamma$ is a $k$-regular graph with $n$ vertices and the line graph $\operatorname{L(\Gamma)}$ admits a perfect $1$-code $C$, then by Equation \ref{equation}, we have
\begin{equation*}
|C|=\frac{\frac{nk}{2}}{2k-1}.
\end{equation*}
Furthermore, every vertex in $C$ is related to an edge in the graph $\Gamma$ and therefore the perfect $1$-code $C$ can be considered as a $1$-regular induced subgraph, say $\overline{C}$, which is also a vertex cover of the graph $\Gamma$. In other words, the line graph $\operatorname{L(\Gamma)}$ contains $|\overline{C}|$ edges with mutually disjoint closed edge neighborhoods. Recall that the closed edge neighborhood of an edge $e$ consists
of the neighborhood of $e$ together with the edge $e$ itself.  Moreover, $-(k-1)$ must be an eigenvalue of the graph $\Gamma$ because $-1$ must be an eigenvalue of the line graph $\operatorname{L(\Gamma)}$ (cf. Obs. \ref{Observation2}).
Therefore we can conclude the following proposition.
\begin{prop} \label{bipartite}
Let $\Gamma$ be a $k$-regular graph such that its line graph $\operatorname{L(\Gamma)}$ admits a perfect $1$-code. Then $-(k-1)$ is an eigenvalue of the graph $\Gamma$. Moreover, if the graph $\Gamma$ is bipartite, then $\pm k$ and $\pm(k-1)$ must be the eigenvalues of the graph $\Gamma$.
\end{prop}
Distance-regular graphs with least eigenvalue $-2$ have been classified as follows.
\begin{thm}\label{thm:drg-2} \cite[Thm.~3.12.4 and 4.2.16]{BCN} Let $\Gamma$ be a distance-regular graph with least eigenvalue $-2$. Then $\Gamma$ is a cycle of even length, or its diameter $d$ equals $2,3,4,$ or $6$. Moreover,
\begin{itemize}
\item If $d=2$, then $\Gamma$ is a cocktail party graph, a triangular graph, a lattice graph, the Petersen graph, the Clebsch graph, the Shrikhande graph, the Schl\"{a}fli graph, or one of the three Chang graphs,
   \item If $d=3$, then $\Gamma$ is the line graph of the Petersen graph, the line graph of the Hoffman-Singleton graph, the line graph of a strongly regular graph with parameters $(3250,57,0,1)$, or the line graph of the incidence graph of a projective plane,
   \item If $d=4$, then $\Gamma$ is the line graph of the incidence graph of a generalized quadrangle of order $(q,q)$,
   \item If $d=6$, then $\Gamma$ is the line graph of the incidence graph of a generalized hexagon of order $(q,q)$.
   \end{itemize}
\end{thm}
We note that the distance-regular graphs with least eigenvalue {\em larger} than $-2$ are also known. Besides the complete graphs (with least eigenvalue $-1$), there are the cycles of odd length. Recall that for a complete graph, perfect $1$-codes are only one-element subsets of the vertex set (see also Rem. \ref{completegraph}). Furthermore, a cycle graph $C_{n}$ of length $n$ has eigenvalues $2\cos(\frac{2\pi j}{n})$, where $j=0,1,\ldots,n-1$. If this graph admits a perfect $1$-code, then $-1$ must be an eigenvalue of this graph by Observation \ref{Observation2}. Therefore we can conclude the following straightforward proposition.
\begin{prop} \label{Cycle}
A cycle graph $C_{n}$ admits a perfect $1$-code if and only if $3$ divides $n$.
\end{prop}
Now we have to investigate four cases $d=2$, $d=3$, $d=4$ and $d=6$ in Theorem \ref{thm:drg-2}. The case $d=2$ can be ruled out by Observation \ref{Observation1}. For $d=3$, the line graph of the Petersen graph is an antipodal distance-regular graph and every fiber of this graph is a perfect $1$-code (cf. Cor. \ref{antipodaldrg3}). Furthermore, the line graph of the Hoffman-Singleton graph, the line graph of a strongly regular graph with parameters $(3250,57,0,1)$, and the line graph of the incidence graph of a projective plane don't admit a perfect $1$-code because $-1$ is not an eigenvalue of these graphs.
\subsection{The line graph of the incidence graph of a generalized quadrangle}
Let $\Gamma$ be the incidence graph of a generalized quadrangle of order $(q,q)$. Then it has intersection array $\{q+1,q,q,q;1,1,1,q+1\}$ and five distinct eigenvalues $\{\pm(q+1),\pm \sqrt{2q},0\}$ (cf. \cite[Sec.~$6.5$]{BCN}). Therefore, by Proposition \ref{bipartite}, the only possible case is $q=2$. If $q=2$, then the graph $\Gamma$ is indeed the Tutte-Coxeter graph. If the line graph of this graph admits a perfect $1$-code $C$, then $|C|=9$ by Equation \ref{equation}, because it is $4$-regular with $45$ vertices. Now, in Figure \ref{Fig}, we consider the Doily representation of a generalize quadrangle of order $(2,2)$ with nine marked flags (i.e., incident point-line pairs) with distinct non-black colours such that no two of them share a point or a line. Then the collection of these nine marked flags represents a perfect $1$-code in the line graph of the Tutte-Coxeter graph since each of its vertices corresponds to a flag of $\operatorname{GQ(2,2)}$. We note that the points of these nine marked flags together with the black lines form a $\operatorname{GQ(2,1)}$ subquadrangle, and their lines together with the black points form a complementary $\operatorname{GQ(1,2)}$ subquadrangle. Now we can conclude the following proposition.
\begin{center}
\begin{figure}[h]

 \centering
 \centerline{\includegraphics*[height=7cm]{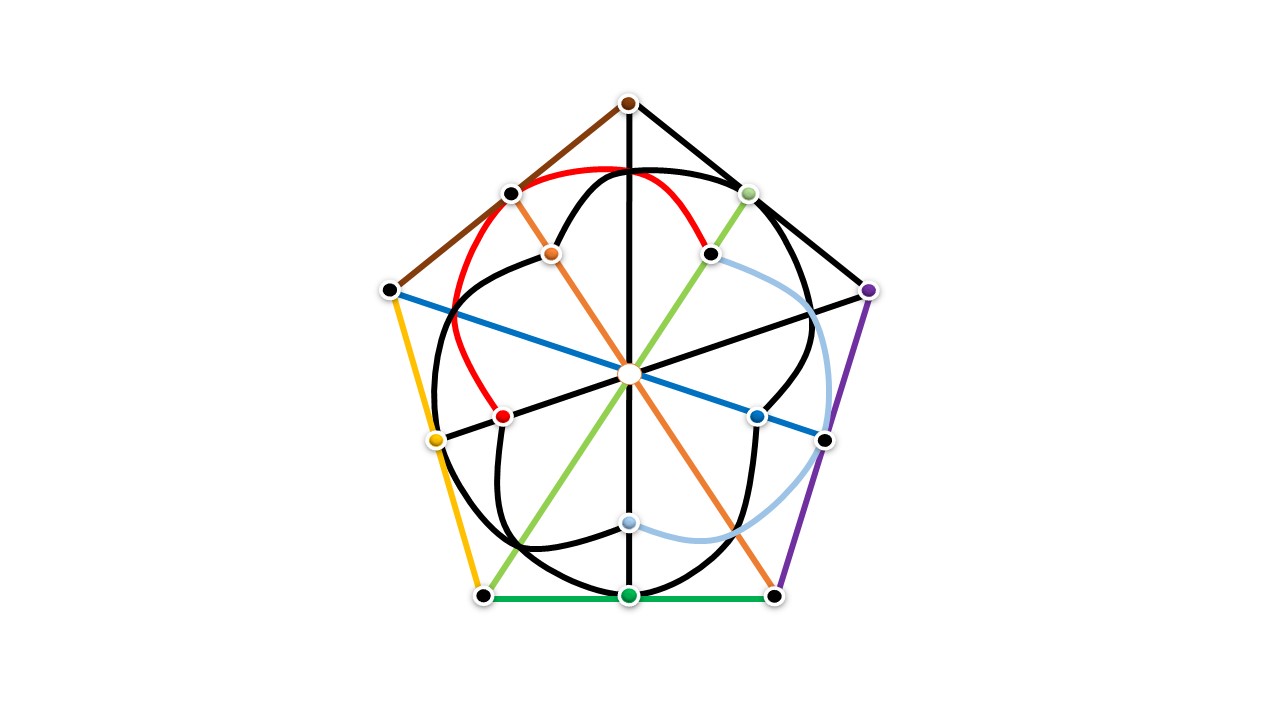}}
 \label{diagram}
 \caption{The Doily representation of $\operatorname{GQ(2,2)}$} \label{Fig}
\end{figure}
\end{center}
\begin{prop} \label{LGQ}
The line graph of the incidence graph of a generalized quadrangle of order $(q,q)$ admits a perfect $1$-code if and only if $q=2$.
\end{prop}
\subsection{The line graph of the incidence graph of a generalized hexagon} \label{SLGH}
Let $\Gamma$ be the incidence graph of a generalized hexagon of order $(q,q)$, where $q \geq 2$. Then it has intersection array $$\{q+1,q,q,q,q,q;1,1,1,1,1,q+1\}$$ and seven distinct eigenvalues $\{\pm(q+1),\pm \sqrt{3q},\pm \sqrt{q},0\}$. If the line graph of this graph admits a perfect $1$-code, then the only possible case is $q=3$ by Proposition \ref{bipartite}. Let $\Gamma$ be the incidence graph of a generalized hexagon of order $(3,3)$. If the line graph $\operatorname{L(\Gamma)}$ admits a perfect $1$-code $C$, then $|C|=208$ by Equation \ref{equation} since this graph is $6$-regular with $1456$ vertices. On the other hand, the code $C$ in the line graph $\operatorname{L(\Gamma)}$ can be related to a code $\overline{C}$ in the graph $\Gamma$ in such a way that $\overline{C}$ is a $1$-regular induced subgraph of the graph $\Gamma$ consisting of $208$ edges with mutually disjoint closed edge neighbourhoods. Additionally, $\overline{C}$ is also a vertex cover of the graph $\Gamma$.

Recall that the only known generalized hexagon of order $(3,3)$ is called the \emph{split Cayley hexagon of order $3$}. Let $\Gamma$ be the incidence graph of the split Cayley hexagon of order $3$. Then it has a lot of substructures which are generalized hexagons of order $(1,3)$ and $(3,1)$ (cf. \cite{DV}). Every generalized hexagon of order $(1,q)$ is isomorphic to the double of a projective plane of order $q$ for which the point set of the double is the set of points and lines and the line set is exactly the set of flags of the projective plane.
\begin{lem} \label{GH(1,q)}
The incidence graph of the double of a projective plane of order $q$ contains at most $q^{2}+q+1$ edges with mutually disjoint closed edge neighbourhoods.
\end{lem}
\begin{proof}
Let $\overline{C_{1}}$ be a collection of edges with mutually disjoint closed edge neighbourhoods of the incidence graph of this structure. If an edge $(p,(p,\ell))$ is in the set $\overline{C_{1}}$, for a point $p$ and line $\ell$, then $(\ell,(p',\ell))$ is not in $\overline{C_{1}}$ for every point $p'$ in $\ell$. Similarly, if an edge $(\ell,(p,\ell))$ is in the set $\overline{C_{1}}$, for a point $p$ and line $\ell$, then $(p,(p,\ell'))$ is not in $\overline{C_{1}}$ for every line $\ell'$ which contains the point $p$. On the other hand, a projective plane of order $q$ contains $q^{2}+q+1$ points and $q^{2}+q+1$ lines which implies that the set $\overline{C_{1}}$ contains at most $q^{2}+q+1$ edges with mutually disjoint closed edge neighbourhoods and the result follows.
\end{proof}

By using the {\sf FinInG} package \cite{fining} for {\sf GAP} \cite{GAP}, the graph $\Gamma$ can be constructed as the incidence graph of a block design%
\footnote{It has $364$ points and $364$ blocks such that each block consists of $4$ points and each point belongs to $4$ blocks.} with the point set $\{1,2,\ldots,364\}$ by the following commands.
\begin{verbatim}
gh:=SplitCayleyHexagon(3);
des:=BlockDesignOfGeneralisedPolygon(gh);
\end{verbatim}
Without loss of generality, let the point $1$ be outside the code $\overline{C}$. Then this vertex belongs to a substructure with the following point set $A$ of size $26$ and block set $B$ of size $52$ consisting of all blocks of this design with the property that each point in $A$ belongs to $4$ blocks in $B$.
\begin{verbatim}
A:={1,2,3,4,5,6,7,8,9,12,15,18,23,28,33,44,65,68,69,88,91,92,129,130,178,179};
B:={[1,2,111,112],[1,28,55,56],[1,65,127,128],[1,88,176,177],[2,3,224,227],
[2,4,225,228],[2,5,223,226],[3,23,45,46],[3,129,245,247],[3,178,311,313],
[4,18,34,35],[4,130,246,248],[4,179,292,312],[5,6,10,11],[5,7,13,14],
[5,8,16,17],[6,44,93,94],[6,91,189,194],[6,92,190,195],[7,33,70,71],
[7,68,140,145],[7,69,141,146],[8,9,21,22],[8,12,26,27],[8,15,31,32],
[9,88,186,191],[9,178,321,327],[9,179,281,293],[12,65,137,142],
[12,129,259,265],[12,130,260,266],[15,18,38,41],[15,23,49,52],
[15,28,59,62],[18,69,156,167],[18,92,203,214],[23,68,155,166],
[23,91,202,213],[28,33,78,83],[28,44,101,106],[33,130,222,279],
[33,163,178,331],[44,129,171,300],[44,179,208,338],[65,68,175,344],
[65,92,218,297],[68,179,241,355],[69,88,210,334],[69,129,244,357],
[88,91,161,306],[91,130,242,361],[92,178,239,356]}
\end{verbatim}
In this substructure, if we consider each block as a line, then it is isomorphic to the generalized hexagon of order $(1,3)$%
\footnote{We double-checked this with {\sf GAP} \cite{GAP}.}. Moreover, each block in $B$ consists of two points in $A$ and two points in the following set $P$ of size $104$.
\begin{verbatim}
P:={10,11,13,14,16,17,21,22,26,27,31,32,34,35,38,41,45,46,49,52,55,56,59,62,
70,71,78,83,93,94,101,106,111,112,127,128,137,140,141,142,145,146,155,156,
161,163,166,167,171,175,176,177,186,189,190,191,194,195,202,203,208,210,213,
214,218,222,223,224,225,226,227,228,239,241,242,244,245,246,247,248,259,260,
265,266,279,281,292,293,297,300,306,311,312,313,321,327,331,334,338,344,355,
356,357,361}
\end{verbatim}
Now consider another substructure with point set $M$ consisting of the points not in $A \cup P$ and block set $N$ consisting of the blocks not in $B$. In this new substructure, if we consider each block as a line, then it is indeed isomorphic to the interesting subgeometry with $234$ points and $312$ lines which is illustrated in \cite[Sec.~$3$]{DV}. By these structures and the properties of the code $\overline{C}$, we could obtain a contradiction as follows. Let $\Gamma_{i}(u)$ denote the set of vertices at distance $i$ from the vertex $u$ in the graph $\Gamma$, where $1\leq i \leq 6$. Recall that without loss of generality, the point $1$ is outside the code $\overline{C}$. We use \emph{inner} and \emph{outer} for elements in $\overline{C}$ and outside $\overline{C}$, respectively. Therefore the vertex $1$ is adjacent to $4$ inner blocks. Moreover, these $4$ blocks are adjacent to $4$ inner points and $8$ outer points in $\Gamma_{2}(1)$ forming the sets $\Gamma_{2}^{in}(1)$ and $\Gamma_{2}^{out}(1)$, respectively. We proceed this approach to find $\Gamma_{6}^{in}(1)$ and $\Gamma_{6}^{out}(1)$. The $4$ inner points of $\Gamma_{2}^{in}(1)$ are adjacent to $12$ outer blocks in $\Gamma_{3}(1)$ forming the set $\Gamma_{3}^{out}(1)$, and the $8$ outer points of $\Gamma_{2}^{out}(1)$ are adjacent to $24$ inner blocks in $\Gamma_{3}(1)$ forming the set $\Gamma_{3}^{in}(1)$. Moreover, the $12$ outer blocks of $\Gamma_{3}^{out}(1)$ are adjacent to $36$ inner points in $\Gamma_{4}(1)$ forming the set $\Gamma_{4}^{in}(1)$, and the $24$ inner blocks of $\Gamma_{3}^{in}(1)$ are adjacent to $24$ inner points and $48$ outer points in $\Gamma_{4}(1)$ forming the sets $\Gamma_{4}^{in'}(1)$ and $\Gamma_{4}^{out}(1)$, respectively. Furthermore, the $36$ inner points of $\Gamma_{4}^{in}(1)$ are adjacent to $36$ inner blocks and $72$ outer blocks in $\Gamma_{5}(1)$ forming the sets $\Gamma_{5}^{in}(1)$ and $\Gamma_{5}^{out}(1)$, respectively, and the $24$ inner points of $\Gamma_{4}^{in'}(1)$ are adjacent to $72$ outer blocks in $\Gamma_{5}(1)$ forming the set $\Gamma_{5}^{out'}(1)$, and the $48$ outer points of $\Gamma_{4}^{out}(1)$ are adjacent to $144$ inner blocks in $\Gamma_{5}(1)$ forming the set $\Gamma_{5}^{in'}(1)$. Finally, the $144$ inner blocks of $\Gamma_{5}^{in'}(1)$ are adjacent to $144$ inner points in $\Gamma_{6}(1)$ forming the set $\Gamma_{6}^{in}(1)$ and the $99$ remaining points in $\Gamma_{6}(1)$ are outer points forming the set $\Gamma_{6}^{out}(1)$ since the code $\overline{C}$ consists of $4+24+36+144=208$ points.

By using {\sf GAP} \cite{GAP}, it turns out that there are $4$ points in $A \cap \Gamma_{2}(1)$ and $8$ points in $P \cap \Gamma_{2}(1)$. Moreover, there are $12$ blocks in $B \cap \Gamma_{3}(1)$ and $24$ blocks in $N \cap \Gamma_{3}(1)$. Furthermore, there are $12$ points in $A \cap \Gamma_{4}(1)$, $24$ points in $P \cap \Gamma_{4}(1)$ and $72$ points in $M \cap \Gamma_{4}(1)$. Moreover, there are $36$ blocks in $B \cap \Gamma_{5}(1)$ and $288$ blocks in $N \cap \Gamma_{5}(1)$. Finally, there are $9$ points in $A \cap \Gamma_{6}(1)$, $72$ points in $P \cap \Gamma_{6}(1)$ and $162$ points in $M \cap \Gamma_{6}(1)$.

Moreover, there exist five cases based on the number of inner points $i$ in $A \cap \Gamma_{2}(1)$, where $0 \leq i \leq 4$. If $|A \cap \Gamma_{2}^{in}(1)|=i$, then $|A \cap \Gamma_{2}^{out}(1)|=|P \cap \Gamma_{2}^{in}(1)|=4-i$, $|P \cap \Gamma_{2}^{out}(1)|=4+i$, $|B \cap \Gamma_{3}^{in}(1)|=|A \cap (\Gamma_{4}^{in'}(1) \cup \Gamma_{4}^{out}(1))|=|N \cap \Gamma_{3}^{out}(1)|=12-3i$, $|N \cap \Gamma_{3}^{in}(1)|=|M \cap \Gamma_{4}^{in'}(1)|=12+3i$, $|B \cap \Gamma_{3}^{out}(1)|=|A \cap \Gamma_{4}^{in}(1)|=|B \cap \Gamma_{5}^{in}(1)|=3i$, $|P \cap \Gamma_{4}^{in}(1)|=|B \cap \Gamma_{5}^{out}(1)|=6i$, $|P \cap (\Gamma_{4}^{in'}(1) \cup \Gamma_{4}^{out}(1))|=24-6i$, $|M \cap \Gamma_{4}^{in}(1)|=36-9i$, and $|M \cap \Gamma_{4}^{out}(1)|=24+6i$. Moreover, as each point of $P \cap \Gamma_{4}(1)$ has the neighbor from $B$ in $\Gamma_{3}(1)$, it follows that each block of $B \cap \Gamma_{5}(1)$ has one neighbor in $A \cap \Gamma_{4}(1)$ and one in $A \cap \Gamma_{6}(1)$.

Now suppose that $i \geq 1$. Then $|A \cap \Gamma_{6}^{in}(1)|=6$ and $|A \cap \Gamma_{6}^{out}(1)|=3$. Since each point of $A \cap \Gamma_{6}^{out}(1)$ has precisely $i$ neighbors in $B \cap \Gamma_{5}^{in}(1)$, it follows that $|B \cap \Gamma_{5}^{in'}(1)|=6+3(4-i)=18-3i$, and then $|A \cap \Gamma_{4}^{out}(1)|=6-i$ and $|A \cap \Gamma_{4}^{in'}(1)|=6-2i$. Therefore, there are $i+3i+(6-2i)+6=12+2i>13$ inner points in $A$, each of which is adjacent to an inner block in $B$, contradicting Lemma \ref{GH(1,q)} (see Figure \ref{Fig2}). This then leaves us with the case $i=0$.

Let there exist $a$ inner and $9-a$ outer points in $A \cap \Gamma_{6}(1)$, where $0 \leq a \leq 9 $. Then there are $3a$ blocks in $B \cap \Gamma_{5}^{out'}(1)$. Moreover, $|A \cap \Gamma_{4}^{in'}(1)|=a$ since each block of $B \cap \Gamma_{5}^{out'}(1)$ has one neighbor in $A \cap \Gamma_{4}^{in'}(1)$, and each point of $A \cap \Gamma_{4}^{in'}(1)$ has three neighbors in $B \cap \Gamma_{5}^{out'}(1)$. If $a \geq 1$, then consider the $(3,2)$-biregular bipartite incidence graph with the point set consisting of the $2a$ inner points in the union of $A \cap \Gamma_{4}^{in'}(1)$ and $A \cap \Gamma_{6}(1)$, and the block set consisting of the $3a$ blocks in $B \cap \Gamma_{5}^{out'}(1)$. This graph contains at least $35$ vertices since the girth of this graph is at least $12$. This implies that $a \geq 7$ and therefore there are at least $14$ inner points in $A$, contradicting Lemma \ref{GH(1,q)}. It follows that the only possible case for $i=0$ is as in Figure \ref{Fig3} and we prove that it is impossible. On one side, each of the $12$ points in $P \cap \Gamma_{4}^{in'}(1)$ has three neighbors in $N \cap \Gamma_{5}^{out'}(1)$ and two points in the set $P$ can not share the same neighbor in $N$. Therefore each block in the set $R$, consisting of the $36$ blocks in $N \cap \Gamma_{5}^{out'}(1)$ which do not have a neighbor in $P \cap \Gamma_{4}^{in'}(1)$, must be adjacent to a unique point in $P \cap \Gamma_{6}^{in}(1)$. Moreover, each of the $36$ points in $P \cap \Gamma_{6}^{in}(1)$ has exactly one neighbor in $N \cap \Gamma_{5}^{out'}(1)$.  To see this, suppose in contrary that there exists a point in $P \cap \Gamma_{6}^{in}(1)$ which has at least two neighbors in $N \cap \Gamma_{5}^{out'}(1)$. Then there exists a point $u \in P \cap \Gamma_{6}^{in}(1)$ which has no neighbor in $R$. Furthermore, there are two blocks in $B \cap \Gamma_{5}^{in'}(1)$ at distance $3$  and therefore two points in $P \cap \Gamma_{6}^{in}(1)$ at distance $4$ form $u$ which are adjacent to at most six blocks in $R$. Moreover, there are at most three blocks in $N \cap \Gamma_{5}^{out}(1)$ at distance $3$ and therefore at most three points in $P \cap \Gamma_{6}^{in}(1)$ at distance $4$ form $u$ which are adjacent to at most six blocks in $R$. Additionally, there are no blocks in $N \cap \Gamma_{3}^{in}(1)$ at distance $3$ from $u$. It follows that there is no path of length at most $6$ from $u$ to some blocks of $R$, a contradiction. This implies that each of the $36$ points in $P \cap \Gamma_{6}^{in}(1)$ has exactly two neighbors in $N \cap \Gamma_{5}^{out}(1)$. On the other side, by using {\sf GAP} \cite{GAP}, it turns out that there are exactly $27$ points in $P \cap \Gamma_{6}^{in}(1)$ which have two neighbors in $\Gamma_{5}^{in}(1) \cup \Gamma_{5}^{out}(1)$%
\footnote{There are $2^{4}=16$ cases for these sets depending on the choice of the four inner points in $P \cap \Gamma_{2}^{in}(1)$. All of these cases have been checked with {\sf GAP} \cite{GAP}.}, a contradiction, and this completes the proof.

\begin{center}
\begin{figure}[h]

 \centering
 \centerline{\includegraphics*[height=12cm]{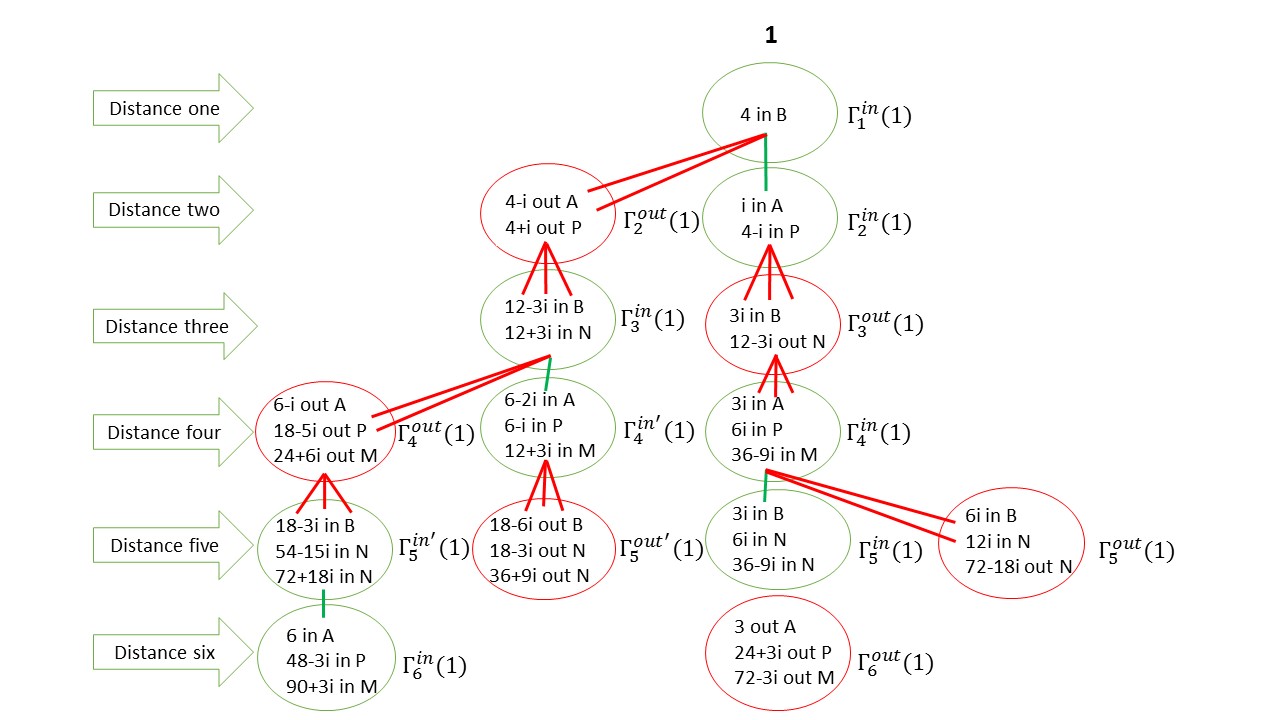}}
 \label{diagram}
 \caption{The incidence graph of $\operatorname{GH(3,3)}$ for $i>0$} \label{Fig2}
\end{figure}
\end{center}

\begin{center}
\begin{figure}[h]

 \centering
 \centerline{\includegraphics*[height=12cm]{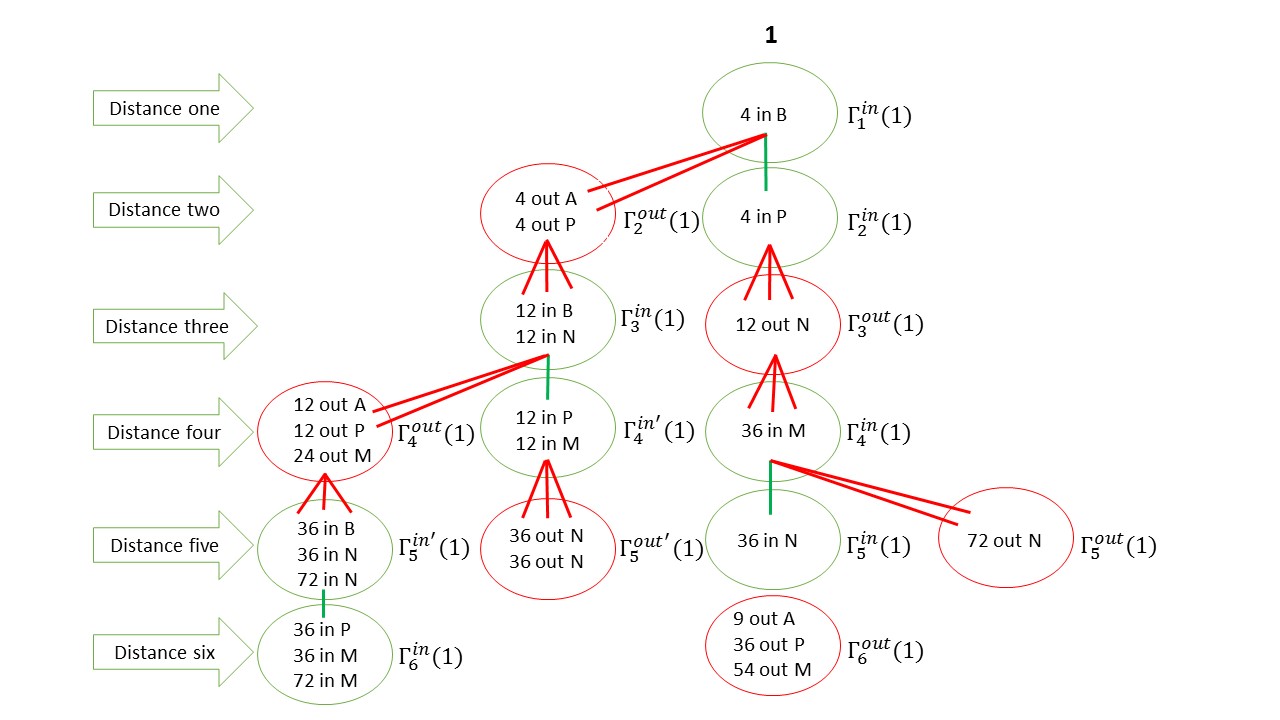}}
 \label{diagram}
 \caption{The incidence graph of $\operatorname{GH(3,3)}$ for $i=0$} \label{Fig3}
\end{figure}
\end{center}

Therefore we can conclude the following proposition.
\begin{prop} \label{LGH}
If the line graph of the incidence graph of a generalized hexagon of order $(q,q)$, where $q \geq 2$, admits a  perfect $1$-code, then $q=3$. Moreover, the incidence graph the split Cayley hexagon of order $3$ doesn't admit a perfect $1$-code.%
\footnote{The anonymous referee double-checked this with {\sf GAP} \cite{GAP}.}
\end{prop}
\section{Distance-regular graphs with small diameter}
As far as we know, the general problem of characterizing distance-regular graphs with small diameter greater than $2$ which admit a perfect $1$-code is hard. Therefore we give an overview up to diameter $4$. For a complete graph, the only perfect $1$-codes are the one-element subsets (cf. Prop.~\ref{completegraph}). Moreover, by Observation \ref{Observation1}, there is no perfect $1$-code in a strongly regular graph.

Now suppose that $C$ is a perfect $1$-code in a distance-regular graph with diameter $3$. Then the distance between two vertices in $C$ must be $3$.
Let $C$ be a perfect $1$-code in a bipartite distance-regular graph with diameter $3$. Then $C$ contains exactly two elements from different parts since the distance between two vertices in $C$ must be $3$. Therefore this graph must be a complete bipartite graph minus a perfect matching. Hence we can conclude the following proposition.
\begin{prop} \label{Bipartitedrg3}
Let $\Gamma$ be a bipartite distance-regular graph with diameter $3$. Then $\Gamma$ admits a perfect $1$-code $C$ if and only if $\Gamma$ is a complete bipartite graph minus a perfect matching and $C$ contains exactly two elements at distance $3$ from different parts.
\end{prop}
Note that every antipodal distance-regular graph of diameter $3$ (including the bipartite ones) admits a perfect $1$-code -- in fact, each fiber is such a code. Now we deal with primitive distance-regular graphs with diameter $3$.
\begin{Observation} \label{Observation3}
If a primitive distance-regular graph $\Gamma$ with diameter $3$ admits a perfect $1$-code, then it has eigenvalue $-1$ by Observation \ref{Observation2} and therefore its distance-$3$ graph is strongly regular (cf. \cite[Prop.~4.2.17]{BCN}). It follows that if this strongly regular graph has parameters $(n,k,\lambda,\mu)$, then $\lambda \geq |C|-2$ because the perfect code $C$ is a clique in the distance-$3$ graph by Observation \ref{Observation1}.
\end{Observation}
Among primitive distance-regular graphs with diameter $3$ and small number of vertices, the first putative example is the Odd graph with $35$ vertices and intersection array $\{4,3,3;1,1,2\}$ since it has eigenvalue $-1$ (cf. \cite[Chap.~$14$]{BCN} and Obs. \ref{Observation2}). This graph admits a perfect $1$-code with $7$ vertices (cf. \cite[Fig.~1]{JV}). The second example is the Sylvester graph with $36$ vertices and intersection array $\{5,4,2;1,1,4\}$. Indeed every $6$-clique in the distance-$3$ graph of the Sylvester graph corresponds to a perfect $1$-code in this graph (see also \cite[Sec.~$3$]{JV2}).
\begin{prop} \label{Sylvester}
The Sylvester graph admits a perfect $1$-code.
\end{prop}

Let $\Gamma$ be a distance-regular graph with diameter $4$. If $\Gamma$ is antipodal, then it doesn't admit a perfect $1$-code (cf. Cor. \ref{antipodaldrg45}). Now suppose that $\Gamma$ is bipartite with degree $k$. Then it has eigenvalues $\{\pm k,0,\pm1\}$ because it has eigenvalue $-1$ by Observation \ref{Observation2}. The following lemma shows there is no such graph.
\begin{lem}
There is no bipartite distance-regular graph with diameter $4$ and eigenvalues $\{\pm k,0,\pm1\}$.
\end{lem}
\begin{proof}
Let $\Gamma$ be a bipartite distance-regular graph with diameter $4$ and eigenvalues $\{\pm k,0,\pm1\}$. We note that this graph can not be the cycle graph on $8$ vertices and therefore $k>2$. On the other hand, if this graph has intersection array $\{k,b_{1},b_{2},b_{3};1,c_{2},c_{3},c_{4}\}$, then by considering the intersection matrix of this graph, the second largest eigenvalue must be $(c_{2}+1)k-c_{2}(c_{3}+1)$ and therefore $(c_{2}+1)k-c_{2}(c_{3}+1)=1$. On the other hand, $c_{2}$ must divide $k$ (cf. \cite[Lem.~$1.7.2$]{BCN}) which implies that $c_{2}=1$ and $c_{3}=2k-2$. Therefore $k \leq 2$ since $c_{3} \leq k$, a contradiction, and this completes the proof.
\end{proof}
Therefore we can conclude the following proposition.
\begin{prop}
A bipartite distance-regular graph with diameter $4$ doesn't admit a perfect $1$-code.
\end{prop}
Among primitive distance-regular graphs with diameter $4$ and small number of vertices, the first putative example is the Coxeter graph with $28$ vertices and intersection array $\{3,2,2,1;1,1,1,2\}$ since it has eigenvalue $-1$ (cf. \cite[Chap.~$14$]{BCN} and Obs. \ref{Observation2}).
This graph partitions into three $7$-gons and a $7$-coclique (cf. \cite{Coxeter}). It turns out that the $7$-coclique is indeed a perfect $1$-code. Therefore we can conclude the following proposition.
\begin{prop} \label{Coxeter}
The Coxeter graph admits a perfect $1$-code.
\end{prop}
\section{Distance-regular graphs with small valency}
All known distance-regular graphs with small valency at most $4$, the distance-regular graphs with known putative intersection arrays for valency $5$, and all distance-regular graphs with girth $3$ and valency $6$ or $7$  are listed in \cite{VJ}. We give an overview of all possible intersection arrays and corresponding graphs, and indicate which of these admit a perfect $1$-code. Note that for each intersection array in Table \ref{tabledrgvalency3} there is a unique distance-regular graph. Moreover, for each intersection array in Table \ref{tabledrgvalency4} there is a unique distance-regular graph, except possibly for the last array, which corresponds to the incidence graph of a generalized hexagon of order $(3,3)$. Furthermore, in Table \ref{tabledrgvalency5}, all known putative intersection arrays for distance-regular graphs with valency $5$ are listed.  All of the graphs in the table are unique, given their intersection arrays, except possibly the incidence graph of a generalized hexagon of order $(4,4)$ (the last case). All distance-regular graphs with valency at most $7$ and girth $3$ (i.e., with triangles) are listed in Table \ref{tabledrgvalency6} besides the ones with valency at most $5$ that we have encountered in the previous tables. For each of the intersection arrays $\{6,3;1,2\}$ and $\{6,4,4;1,1,3\}$, there are exactly two distance-regular graphs (as mentioned in the table). By $n$, $d$, and $g$, we denote the number of vertices, diameter, and girth, respectively. We note that in the reference column, only one reason is stated.
\subsection{The point graphs of the generalized hexagons of order $(2,2)$} \label{GH(2,2)}
Up to isomorphism there are exactly two generalized hexagons of order $(2,2)$. Each of them is the dual of the other (cf. \cite[Theorem~1]{CT}). Their point graphs (collinearity graphs) give rise to two distance-regular graphs with intersection array $\{6,4,4;1,1,3\}$ and $63$ vertices. We can distinguish the two graphs by whether the graph induced on the vertices at distance $3$ from a fixed vertex is connected or not. These two non-isomorphic distance-regular graphs have been constructed in {\sf GAP} \cite{GAP} with {\sf Grape} \cite{Sgrape} package as {\sf Graph $1$} and {\sf Graph $2$} in \cite{drgorg}. Indeed the graph induced on the vertices at distance $3$ from a fixed vertex in the {\sf Graph $2$} is connected whereas in the {\sf Graph $1$} is disconnected. Let $\Gamma$ be one of such distance-regular graphs which admits a perfect $1$-code $C$. Then $|C|=9$ by Equation \ref{equation}. Moreover, each pair of vertices in $C$ is at distance $3$ by Observation \ref{Observation1} since $\Gamma$ has diameter $3$. Therefore a perfect $1$-code $C$ can be viewed as a clique in the distance-$3$ graph of $\Gamma$. By using {\sf GAP} \cite{GAP}, it turns out that the {\sf Graph $1$} of \cite{drgorg} admits perfect $1$-codes%
\footnote{The {\sf Graph $1$} of \cite{drgorg} with vertex set $\{1,2,\ldots,63\}$ admits $\{3,4,5,7,35,37,42,50,63\}$ and $\{1,2,22,24,30,33,41,61,63\}$ as perfect $1$-codes.} but the {\sf Graph $2$} of \cite{drgorg} doesn't admit a perfect $1$-code since there is no complete subgraph of size $9$ in its distance-$3$ graph. This shows that we can not deduce that a distance-regular graph admits a perfect $1$-code from its spectrum.

\begin{center}
\begin{table}[h]
\begin{tabular}{ l r r c l  c l}
  \hline
  Intersection array & $n$ & $d$ & $g$ & Name & Perfect $1$-code & Reference\\
  \hline
  \{3;1\}& 4 & 1 & 3&$\operatorname{K_{4}}$ & Yes & Rem.~\ref{completegraph}\\
  \{3,2;1,3\} & 6 & 2 &4& $\operatorname{K_{3,3}}$ & No & Eq.~\ref{equation}\\
  \{3,2,1;1,2,3\} & 8 & 3 &4& $\operatorname{K_{3,3}^*}$ & Yes & Prop.~\ref{Bipartitedrg3}\\
  \{3,2;1,1\} & 10 & 2 &5& $\operatorname{Petersen}$ & No & Eq.~\ref{equation}\\
  \{3,2,2;1,1,3\} & 14 & 3 &6& $\operatorname{Heawood}$ & No & Eq.~\ref{equation}\\
  \{3,2,2,1;1,1,2,3\} & 18 & 4 &6& $\operatorname{Pappus}$ & No & Eq.~\ref{equation}\\
  \{3,2,2,1,1;1,1,2,2,3\} & 20 & 5 &6& $\operatorname{Desargues}$ & No & Cor.~\ref{antipodaldrg45}\\
  \{3,2,1,1,1;1,1,1,2,3\} & 20 & 5 &5& $\operatorname{Dodecahedron}$ & No & Cor.~\ref{antipodaldrg45}\\
  \{3,2,2,1;1,1,1,2\}& 28 & 4 &7& $\operatorname{Coxeter}$ & Yes & Prop.~\ref{Coxeter} \\
  \{3,2,2,2;1,1,1,3\} & 30 & 4 &8& $\operatorname{Tutte's}$ $\operatorname{8-cage}$ & No & Eq.~\ref{equation}\\
  \{3,2,2,2,2,1,1,1;& 90 & 8 & 10& $\operatorname{Foster}$ & No & Eq.~\ref{equation}\\
   \hspace{.5cm}1,1,1,1,2,2,2,3\}  & && && &\\
  \{3,2,2,2,1,1,1; & 102 & 7 &9& $\operatorname{Biggs-Smith}$ & No & Eq.~\ref{equation}\\
  \hspace{.5cm}1,1,1,1,1,1,3\}  && & & &&\\
  \{3,2,2,2,2,2; & 126 & 6 &12& $\operatorname{Tutte's}$ $\operatorname{12-cage}$ & No & Eq.~\ref{equation}\\
   \hspace{.5cm}1,1,1,1,1,3\}  & && & &&\\
  \hline
\end{tabular}
\caption{Distance-regular graphs with valency $3$}\label{tabledrgvalency3}
\end{table}
\end{center}
\begin{table}[h]
\begin{center}
\begin{tabular}{ l r r c l c l}
\hline
   Intersection array & $n$ & $d$ & $g$& Name & Perfect $1$-code & Reference\\
  \hline
  \{4;1\} & 5 & 1 &3& $\operatorname{K_{5}}$ & Yes & Rem.~\ref{completegraph}\\
  \{4,1;1,4\} &  6 & 2 &3& $\operatorname{K_{2,2,2}}$ & No & Eq.~\ref{equation}\\
  \{4,3;1,4\} &  8 & 2 &4& $\operatorname{K_{4,4}}$ & No & Eq.~\ref{equation}\\
  \{4,2;1,2\} &  9 & 2 &3& $\operatorname{Paley}$ $\operatorname{graph}$ $\operatorname{P(9)}$ & No & Eq.~\ref{equation}\\
  \{4,3,1;1,3,4\} &  10 & 3 &4& $\operatorname{K^{*}_{5,5}}$ & Yes & Prop.~\ref{Bipartitedrg3}\\
  \{4,3,2;1,2,4\} &  14 & 3 &4& $\operatorname{IG(7,4,2)}$ & No & Eq.~\ref{equation}\\
  \{4,2,1;1,1,4\} &  15 & 3 &3& $\operatorname{L(Petersen)}$ & Yes & Cor.~\ref{antipodaldrg3}\\
  \{4,3,2,1;1,2,3,4\} &  16 & 4 &4& $\operatorname{Q_{4}}$ & No & Eq.~\ref{equation}\\
  \{4,2,2;1,1,2\} &  21 & 3 &3& $\operatorname{L(Heawood)}$ & No & Eq.~\ref{equation}\\
   \{4,3,3;1,1,4\} &  26 & 3 &6& $\operatorname{IG(13,4,1)}$ & No & Eq.~\ref{equation}\\
   \{4,3,3,1;1,1,3,4\} &  32 & 4 &6& $\operatorname{IG(A(2,4)\setminus pc)}$ & No & Eq.~\ref{equation}\\
   \{4,3,3;1,1,2\} &  35 & 3 &6& $\operatorname{O_4}$ & Yes & \cite[Fig.~1]{JV}\\
   \{4,2,2,2;1,1,1,2\} &  45 & 4 &3& $\operatorname{L(Tutte's}$ $\operatorname{8-cage)}$
   & Yes & Prop.~\ref{LGQ}\\
   \{4,3,3,2,2,1,1; &  70 & 7 &6& $\operatorname{DO_4}$ & Yes & Cor.~\ref{DoubleOdd}\\
    \hspace{.5cm}1,1,2,2,3,3,4\} & & & && &\\
   \{4,3,3,3;1,1,1,4\} &  80 & 4 &8& $\operatorname{IG(GQ(3,3))}$ & No & Obs.~\ref{Observation2}\\
   \{4,2,2,2,2,2; &  189 & 6 &3& $\operatorname{L(Tutte's}$ $\operatorname{12-cage})$
   & No & Eq.~\ref{equation}\\
     \hspace{.5cm}1,1,1,1,1,2\} & & & && &\\
   \{4,3,3,3,3,3; &  728 & 6 &12& $\operatorname{IG(GH(3,3))}$ & No & Eq.~\ref{equation}\\
     \hspace{.5cm}1,1,1,1,1,4\}  && & && &\\
  \hline
\end{tabular}
 \caption{Distance-regular graphs with valency $4$}\label{tabledrgvalency4}
\end{center}
\end{table}
\begin{table}[h]
\begin{center}
\begin{tabular}{ l r r c l c l}
  \hline
   Intersection array & $n$ & $d$ &$g$& Name & Perfect $1$-code & Reference\\
  \hline
 \{5;1\} & 6 & 1 &3& $\operatorname{K_{6}}$ & Yes & Rem.~\ref{completegraph}\\
 \{5,4;1,5\} &  10 & 2 &4& $\operatorname{K_{5,5}}$ & No & Eq.~\ref{equation}\\
 \{5,2,1;1,2,5\} &  12 & 3 &3& $\operatorname{Icosahedron}$ & Yes & Cor.~\ref{antipodaldrg3}\\
 \{5,4,1;1,4,5\} &  12 & 3 &4 &$\operatorname{K^{*}_{6,6}}$ & Yes & Prop.~\ref{Bipartitedrg3}\\
 \{5,4;1,2\} &  16 & 2 &4& $\operatorname{Folded}$ $\operatorname{5-cube}$ & No & Eq.~\ref{equation}\\
 \{5,4,3;1,2,5\} &  22 & 3 &4&$\operatorname{IG(11,5,2)}$ & No & Eq.~\ref{equation}\\
 \{5,4,3,2,1;1,2,3,4,5\} &  32 & 5 &4& $\operatorname{Q_{5}}$ & No & Eq.~\ref{equation}\\
 \{5,4,1,1;1,1,4,5\} &  32 & 4 &5& $\operatorname{Armanios-Wells}$ & No & Eq.~\ref{equation}\\
 \{5,4,2;1,1,4\} &  36 & 3 &5& $\operatorname{Sylvester}$ & Yes & Prop.~\ref{Sylvester}\\
  \{5,4,4;1,1,5\} &  42 & 3 &6& $\operatorname{IG(21,5,1)}$ & No & Obs.~\ref{Observation2}\\
  \{5,4,4,1;1,1,4,5\} &  50 & 4 &6& $\operatorname{IG(A(2,5)\setminus pc)}$ & No & Eq.~\ref{equation}\\
  \{5,4,4,3;1,1,2,2\} &  126 & 4 &6& $\operatorname{O_{5}}$ & No & Obs.~\ref{Observation2}\\
  \{5,4,4,4;1,1,1,5\} &  170 & 4 &8& $\operatorname{IG(GQ(4,4))}$ & No & Eq.~\ref{equation}\\
  \{5,4,4,3,3,2,2,1,1; &  252 & 9 &6& $\operatorname{DO_{5}}$ & No & Cor.~\ref{DoubleOdd}\\
   \hspace{.5cm}1,1,2,2,3,3,4,4,5\}  && & && &\\
  \{5,4,4,4,4,4; &  2730 & 6 &12& $\operatorname{IG(GH(4,4))}$ & No & Obs.~\ref{Observation2}\\
  \hspace{.5cm}1,1,1,1,1,5\}  && & && &\\
  \hline
\end{tabular}
 \caption{Distance-regular graphs with valency $5$}\label{tabledrgvalency5}
\end{center}
\end{table}
\begin{table}[h]
\begin{center}
\begin{tabular}{ l r r c l c l}
  \hline
   Intersection array & $n$ & $d$ &$g$& Name & Perfect $1$-code & Reference\\
  \hline
 \{6;1\} & 7 & 1 &3& $\operatorname{K_{7}}$ & Yes & Rem.~\ref{completegraph}\\
 \{6,1;1,6\} &  8 & 2 &3& $\operatorname{K_{2,2,2,2}}$ & No & Eq.~\ref{equation}\\
 \{6,2;1,6\} &  9 & 2 &3& $\operatorname{K_{3,3,3}}$ & No & Eq.~\ref{equation}\\
 \{6,2;1,4\} &  10 & 2 &3& $\operatorname{T(5)}$ & No & Eq.~\ref{equation}\\
 \{6,3;1,3\} &  13 & 2 &3& $\operatorname{P(13)}$ & No & Eq.~\ref{equation}\\
 \{6,4;1,3\} &  15 & 2 &3& $\operatorname{\overline{T(6)}\sim GQ(2,2)}$ & No & Eq.~\ref{equation}\\
 \{6,3;1,2\} &  16 & 2 &3& $\operatorname{L_{2}(4)}$, $\operatorname{Shrikhande}$ & No & Eq.~\ref{equation}\\
 \{6,4,2;1,2,3\} &  27 & 3 &3& $\operatorname{H(3,3)}$ & No & Eq.~\ref{equation}\\
 \{6,4,2,1;1,1,4,6\} &  45 & 4 &3& $\operatorname{halved}$ $\operatorname{Foster}$ & No & Eq.~\ref{equation}\\
 \{6,3,3;1,1,2\} &  52 & 3 &3& $\operatorname{L(IG(13,4,1)}$ & No & Eq.~\ref{equation}\\
 \{6,4,4;1,1,3\} &  63 & 4 &3& $\operatorname{GH(2,2)}$ $\operatorname{(Graph~1)}$& Yes & Sec.~\ref{GH(2,2)}\\
 \{6,4,4;1,1,3\} &  63 & 4 &3& $\operatorname{GH(2,2)}$ $\operatorname{(Graph~2)}$& No & Sec.~\ref{GH(2,2)}\\
 \{6,3,3,3;1,1,1,2\} &  160 & 4 &3& $\operatorname{L(IG(GQ(3,3)))}$ & No & Prop.~\ref{LGQ}\\
 \{6,3,3,3,3,3;1,1,1,1,1,2\} &  1456 & 6 &3& $\operatorname{L(IG(GH(3,3)))}$ & No & Prop.~\ref{LGH}\\
 \hline
  \{7;1\} & 8 & 1 &3& $\operatorname{K_{8}}$ & Yes & Rem.~\ref{completegraph}\\
 \{7,4,1;1,2,7\} &  24 & 3 &3& $\operatorname{Klein}$ & Yes & Cor.~\ref{antipodaldrg3}\\
 \hline
\end{tabular}
 \caption{Distance-regular graphs with girth $3$ and valency $6$ or $7$}\label{tabledrgvalency6}
\end{center}
\end{table}
\section*{Acknowledgements}

\noindent The author would like to thank the anonymous referee for his/her invaluable comments which led to fixing some errors in the proof of Proposition \ref{LGH} and the statement of Proposition \ref{Sylvester}, and improving the presentation of this paper. The author is grateful to the Research Council of Shahid Chamran University of Ahvaz for financial support (SCU.MM99.29248).

\end{document}